\documentclass{amsart}
\usepackage{mystyle}

\newtheorem{Theorem}{Theorem}[section] 

\newtheorem{Definition}{Definition}[section]
\newtheorem{Lemma}{Lemma}[section]
\newtheorem{Proposition}{Proposition}[section]
\newtheorem{Remark}{Remark}[section]

\begin{document}
\title[Stochastic Perron's Method]{Stochastic Perron's method and verification without smoothness using viscosity comparison: the linear case}

\author{ Erhan Bayraktar}
\address{University of Michigan, Department of Mathematics, 530 Church Street, Ann Arbor, MI 48109.} \email{erhan@umich.edu.}
 \thanks{The research of E. Bayraktar was supported in part by the National Science Foundation under grants DMS 0906257 and DMS 0955463.}

 \author{Mihai S\^{\i}rbu}
 \address{University of Texas at Austin,
    Department of Mathematics, 1 University Station C1200, Austin, TX,
    78712.}  
    \email{sirbu@math.utexas.edu.} 
    \thanks{The research of
    M. S\^{\i}rbu was supported in part by the National Science
    Foundation under Grant
    DMS 0908441. 
    }
    \thanks{The authors would like to thank Gerard Brunick and Gordan \v{Z}itkovi\v{c} for their comments. Special thanks go to Ioannis Karatzas for his suggestions that led to  an  improved final version.}
\date{\today}

  

\begin{abstract}
We introduce a stochastic  version of  the classical Perron's method to construct viscosity solutions to linear parabolic equations associated to stochastic differential equations. Using this method, we construct easily two viscosity (sub and super) solutions that squeeze in between  the expected payoff. If a comparison result holds true, then there exists a  unique viscosity solution which   is a  martingale along the solutions of the stochastic differential equation. The unique viscosity solution is actually equal to the expected payoff. This amounts to a verification result (It\^ o's Lemma) for non-smooth viscosity solutions of the linear parabolic equation. 

\end{abstract}

\maketitle 
%
%

 \section{Introduction} 
 The best way to approach
 a Feynman-Kac equation (or a Hamilton-Jacobi-Bellman equation in the case of stochastic control) is to prove existence of a smooth solution and then use It\^o's Lemma to relate it to the corresponding probabilistic representation. 
 In case  such a smooth solution does not exist,  a large number of results in the literature consist in taking the expected payoff (value function) associated to a Markov diffusion and then checking the viscosity solution property. Such an approach essentially uses the  Markov property of the diffusion. If uniqueness in law of the stochastic differential equation does not hold, then a Markov selection is needed to obtain a viscosity solution this way.
 If, in addition,  a viscosity  comparison result holds true (which is  a  purely analytical result), then the conclusion is that the expected payoff (value function) is actually the \emph{unique} viscosity solution.
 
 On the other hand, Ishii \cite{ishii} refined the classical Perron's method to  the case of viscosity solutions. This amounts to a very powerful   \emph{analytical method} to  construct (therefore proving existence) viscosity solutions in very general frameworks. However, if one wants to compare such a viscosity solution obtained  by Perron's method with the expected pay-off (value function),  then one still needs the viscosity property for the expected pay-off (value function). In other words, the program described in the beginning still needs to be carried out.  
 
 In this note, we propose a \emph{stochastic } alternative to Perron's method to construct viscosity solutions, namely Theorem \ref{Perron}. 
 More precisely, we consider the infimum of stochastic super-solutions or the supremum of stochastic sub-solutions to a linear parabolic PDE. By stochastic sub and super-solutions we mean obvious generalizations of the seminal notion of stochastic solution introduced by Stroock and Varadhan \cite{Stroock-Varadhan}. To the best of our knowledge, such a technique does not exist in the literature.
 While this  construction does not provide a \emph{stochastic} solution, it does provide a (weaker) viscosity solution.
 The main advantage of our method is that comparison between such constructed viscosity solutions and the expected pay-off (value function) becomes trivial (see Lemma \ref{lem:verification}), because it is \emph{imbedded in the stochastic definition}. 
  In other words, one does not need to prove any property for the expected pay-off (value function) in order to compare it to the viscosity solution(s) constructed by stochastic Perron's method. Using this result, \emph{if}  viscosity comparison holds, then one gets that the \emph{unique viscosity solution } is actually equal to the expected pay-off (value function) \emph{for free}. The unique viscosity solution is a martingale along any solution of the stochastic differential equation, i.e. is a stochastic solution in the sense of Stroock and Varadhan \cite{Stroock-Varadhan}. This actually amounts to a verification result for non-smooth viscosity solutions, where we can use   uniqueness of viscosity solutions as a substitute for verification.
 
In the present note  we illustrate  these ideas in the simplest framework of \emph{linear} parabolic equations with terminal conditions on the whole state space (a particular version of Feynman-Kac). 
 However, we claim that thess ideas carry over to  much more  general frameworks.
 In particular, other linear cases including   infinite horizons, running-costs, exit times or even reflections on the boundary can be easily treated in an identical way.
 More interestingly, we intend to carry over these ideas to the more important case of Hamilton-Jacobi-Bellman equations associated to  stochastic control and stochastic games (Isaac's equations). 
 These more technical details are left to future work and will be presented in forthcoming papers.

   \section{The set-up and main results}
 
 Fix a time interval $T>0$ and for each $0\leq s<T$ and $x\in \mathbb{R}^d$ consider the stochastic differential equation
 \begin{equation}\label{sde}\left \{
 \begin{array}{l}
 dX_t=b(t,X_t)dt+\sigma (t, X_t)dW_t, \ \ s\leq t\leq T\\
 X_s=x.
 \end{array}\right.
 \end{equation}
 We assume that  the coefficients $b :[0,T]\times \mathbb{R}^d\rightarrow \mathbb{R}^d$ and $\sigma :[0,T]\times \mathbb{R}^d\rightarrow \mathbb{M}_{d,d'}(\mathbb{R})$ are continuous.  We also assume that, for each $(s,x)$ equation \eqref{sde} has at least a weak  non-exploding solution 
 $$\Big((X^{s,x}_t)_{s\leq t\leq T}, (W^{s,x}_t)_{s\leq t\leq T}, \Omega ^{s,x}, \mathcal{F}^{s,x}, \mathbb{P}^{s,x},( \mathcal{F}^{s,x}_t) _{s\leq t\leq T}\Big),$$
  where the  $W^{s,x}$ is  a $d'$-dimensional Brownian motion on the stochastic basis $$(\Omega ^{s,x}, \mathcal{F}^{s,x}, \mathbb{P}^{s,x},( \mathcal{F}^{s,x}_t) _{s\leq t\leq T})$$ and  the filtration $( \mathcal{F}^{s,x}_t) _{s\leq t\leq T}$  satisfies the  usual conditions.
  We denote by $\mathcal{X}^{s,x}$ the non-empty set of such weak solutions.
 It is well known, for example from \cite{MR2190038}, that  a sufficient condition for the existence of non-exploding solutions, in addition to continuity of the coefficients, is the condition of linear growth:
 $$|b(t,x)|+|\sigma (t,x)|\leq C(1+|x|), \ \ (t,x)\in [0,T]\times \mathbb{R}^d.$$ 
 We emphasize that we do \emph{not} assume uniqueness in law of the weak solution.
  \begin{Remark}\label{set-theory}
 Actually, in order to insure that $\mathcal{X}^{s,x}$ is a set in the sense of axiomatic set theory, one should restrict to weak solutions where the probability space $\Omega$ is an element of a fixed universal set $\mathcal{S}$ of possible probability spaces.
 \end{Remark}
\noindent  Now, for some fixed \emph{bounded} and measurable function $g:\mathbb{R}^d\rightarrow \mathbb{R}$, we denote by 
 $$v_*(s,x):=\inf _{\mathcal{X}^{s,x}}\mathbb{E}^{s,x}[g(X^{s,x}_T)], \textrm{~and~} v^*(s,x):=\sup _{\mathcal{X}^{s,x}}\mathbb{E}^{s,x}[g(X^{s,x}_T)].$$ 
We will call the functions $v_*$ and $v^*$ the lower  and the upper expected pay-offs (value functions). It is obvious that
 $$v_*\leq v^*$$ and the two functions coincide if the stochastic differential equation \eqref{sde} has a unique in law weak solution.
 \begin{Remark}\label{measurability} At this stage, we cannot even conclude that  $v_*$ and $v^*$ are  measurable.
 \end{Remark}
We expect that the expected pay-offs (value functions) $v_*$ and $v^*$ be  associated to the following linear PDE:
\begin{equation}\label{pde}
\left \{
\begin{array}{l}
-u_t-L_tu=0\\
u(T,x)=g(x),
\end{array}\right.\end{equation}
where the time dependent operator $L_t$ is defined by
$$(L_t u)(x)=\langle b(t,x),\nabla u(t,x)\rangle +\frac 12 Tr (\sigma (t,x)\sigma ^T (t,x) u_{xx}(t,x)),\ \ 0\leq t<T,\  x\in \mathbb{R}^d.$$ 
 \subsection{Stochastic Perron's method} Let $g:\mathbb{R}^d \rightarrow \mathbb{R}$ be measurable and bounded.  As mentioned in the Introduction,  we now 
 introduce the sets of stochastic super and sub-solutions of the parabolic PDE  \eqref{pde} in the spirit of \cite{Stroock-Varadhan}.
 \begin{Definition}\label{def:supersolutions}
The set of stochastic super-solutions of the parabolic PDE \eqref{pde}, denoted by $\mathcal{U}^+$, is  the set of functions $u:[0,T]\times \mathbb{R}^d\rightarrow \mathbb{R}$ which have the following properties
\begin{enumerate}
\item are upper semicontinuous (USC)  and  bounded on $[0,T]\times \mathbb{R}^d$. In addition,  they satisfy the terminal condition   $u(T,x)\geq g(x)$ for all $x\in \mathbb{R}^d$.

\item 
 for each $(s,x)\in[0,T]\times \mathbb{R}^d$, and each weak solution 
  $$\Big((X^{s,x}_t)_{s\leq t\leq T}, (W^{s,x}_t)_{s\leq t\leq T}, \Omega ^{s,x}, \mathcal{F}^{s,x}, \mathbb{P}^{s,x},( \mathcal{F}^{s,x}_t) _{s\leq t\leq T}\Big)\in \mathcal{X}^{s,x},$$ 
  the process $(u(t,X^{s,x}_t))_{s\leq t\leq T}$ is a supermartingale on $(\Omega ^{s,x}, \mathbb{P}^{s,x})$ with respect to the filtration 
$( \mathcal{F}^{s,x}_t) _{s\leq t\leq T}$.
\end{enumerate}
\end{Definition}
\begin{Definition}\label{def:subsolutions}
The set of stochastic sub-solutions of the parabolic PDE \eqref{pde}, denoted by $\mathcal{U}^-$ , is the  set of functions $u:[0,T]\times \mathbb{R}^d\rightarrow \mathbb{R}$ which have the following properties
\begin{enumerate} 
\item are lower semicontinuous (LSC) and  bounded on $[0,T]\times \mathbb{R}^d$. In addition,  they satisfy the terminal condition   $u(T,x)\leq g(x)$ for all $x\in \mathbb{R}^d$.

\item 
 for each $(s,x)\in[0,T]\times \mathbb{R}^d$, and each weak solution 
  $$\Big((X^{s,x}_t)_{s\leq t\leq T}, (W^{s,x}_t)_{s\leq t\leq T}, \Omega ^{s,x}, \mathcal{F}^{s,x}, \mathbb{P}^{s,x},( \mathcal{F}^{s,x}_t) _{s\leq t\leq T}\Big)\in \mathcal{X}^{s,x},$$ 
  the process $(u(t,X^{s,x}_t))_{s\leq t\leq T}$ is a submartingale on $(\Omega ^{s,x}, \mathbb{P}^{s,x})$ with respect to the filtration 
$( \mathcal{F}^{s,x}_t) _{s\leq t\leq T}$.
\end{enumerate}
\end{Definition}

\begin{Remark}\label{path-regularity}
In the Definitions \ref{def:supersolutions}, \ref{def:subsolutions} of $\mathcal{U}^+, \mathcal{U}^-$ we do not assume that the processes
$(u(t,X^{s,x}_t))_{s\leq t\leq T}$ have (RC) right-continous paths. For this reason, care must be taken when one tries to apply the Optional Sampling Theorem in the form of ``a stopped martingale is a martingale". More precisely, such a theorem holds only with respect to discrete-valued stopping times.
\end{Remark}
\begin{Remark}\label{nonempty}
Since $g$ is assumed bounded, the sets $\mathcal{U}^-$ and $\mathcal{U}^+$  are easily seen to be  non-empty. More precisely any constant function $u(t,x)\equiv k$  which is an upper bound to $g$  ($g\leq k$) is in $\mathcal{U}^+$ and any constant function  $u(t,x)\equiv k$ which is a lower bound to $g$  ($k\leq g$) is in $\mathcal{U}^-$. If one wants to account for a larger class of functions $g$ than bounded, then the definitions of $\mathcal{U}^-$ and $\mathcal{U}^+$ should be changed appropriately, and an assumption on non-emptiness  of $\mathcal{U}^-$ and $\mathcal{U}^+$ should be made.
\end{Remark}
Using the properties of sub(super)-martingales as well as the definition of  $v_*$ and $v^*$, we easily obtain the following result.
\begin{Lemma}\label{lem:verification} For each $u\in \mathcal{U}^-$ and each $w\in \mathcal{U}^+$  we have 
$u\leq v_*\leq v^*\leq w.$
\end{Lemma}
Using the Remark \ref{nonempty} and Lemma \ref{lem:verification},  we can  define
$$v^-:= \sup _{u\in \mathcal{U}^-}u\leq v_*\leq v^*\leq v^+:=\inf _{w\in \mathcal{U}^+}w.$$
\begin{Lemma}\label{closure}
We have $v^-\in \mathcal{U}^-,\ \ \ \ v^+\in \mathcal{U}^+.$\end{Lemma}
\begin{proof}\mbox{} It is well known that an infimum of upper semicontinous functions is upper semicontinuous. While we cannot conclude directly that the point-wise  infimum of supermartingales is a supermartingale 
(because the set of supermartingales may be \emph{uncountable}, and the use of \emph{essential infimum} would be needed), we can appeal to Proposition  \ref{countable} in the Appendix 
and conclude that $v^+$ is actually the point-wise infimum of a \emph{countable} set of functions  $w_n\in \mathcal{U}^+$, $n=1,2.\dots.$ Now, the point-wise infimum of a \emph{countable} set of supermartingales is, indeed, a supermartingale. 
 The terminal condition for $v^+$  is satisfied and the boundedness follows easily since $g$ is bounded  so $v^*$ is bounded, and using Remark \ref{nonempty} and Lemma \ref{lem:verification} we have 
$$v^*\leq v^+\leq \sup_{x\in \mathbb{R}^d} g(x).$$ Therefore   $v^+\in {U}^+$. The other part is identical.
\end{proof}

\begin{Remark}\label{easy-Perron} Using Lemma \ref{closure}, one could  easily show that $v^+$ is a viscosity supersolution of \eqref{pde}  (i.e. satisfies \eqref{pde4} below in the viscosity sense) and $v^-$ is a viscosity subsolution of \eqref{pde} (i.e. satisfies \eqref{pde3} below in the viscosity sense).
 However, while true, this does not present much interest.
\end{Remark}
The following is the main technical result of the present note:
\begin{Theorem}{\rm (Stochastic Perron's Method)}\label{Perron}
If $g$ is bounded and LSC  then 
$v^-$ is a bounded and LSC viscosity  supersolution of 
\begin{equation}\label{pde1}
\left \{
\begin{array}{l}
-u_t-L_tu\geq 0,\\
u(T,x)\geq g(x).
\end{array}\right.\end{equation}
If $g$ is bounded and USC then  $v^+$ is a bounded and  USC  viscosity subsolution of 
\begin{equation}\label{pde2}
\left \{
\begin{array}{l}
-u_t-L_tu\leq 0,\\
u(T,x)\leq g(x).
\end{array}\right.\end{equation}\end{Theorem}
\begin{Remark}\label{terminal-condition} We have $v^-(T,x)\leq g(x)$ and $v^+(T,x)\geq g(x)$ by construction.  Therefore, the terminal conditions in \eqref{pde1} and \eqref{pde2} can be replaced by equalities.
\end{Remark}
\begin{Remark}\label{viscosity}We would like to point out that the semi-continuous solutions obtained by Perron's method have the correct semicontinuity needed for such a definition. We refer the reader to \cite{CIL} for an introduction to (semicontinuous) viscosity sub and super-solutions of second order equations.
\end{Remark}
\begin{proof}\mbox{} 
We will only prove that $v^+$ is a subsolution of \eqref{pde2}: the other part is symmetric. 

\noindent {\bf Step 1}. \emph{The interior sub-solution property.} Note that we already know that $v^+$ is bounded and upper semicontinuous. Let 
$$\varphi:[0,T]\times \mathbb{R}^d\rightarrow \mathbb{R}$$ be a $C^{1,2}$-test function function and assume that 
$v^+-\varphi$ attains a strict local maximum (an assumption which is not restrictive) equal to zero at some interior point $(t_0, x_0)\in (0,T)\times \mathbb{R}^d$. Assume that $v^+$ does not satisfy the viscosity subsolution property, and therefore 
$$-\varphi _t(t_0,x_0)-\mathcal{L}_t \varphi (t_0, x_0)>0.$$
Since the coefficients of the SDE  are continuous, we conclude that there exists a small enough ball
$B(t_0, x_0, \varepsilon)$ such that
$$-\varphi _t-\mathcal{L}_t \varphi >0\textrm{~on~} \overline{ B(t_0, x_0, \varepsilon)},$$
and 
$$\varphi > v^+\textrm{~on~} \overline{ B(t_0, x_0, \varepsilon)}-(t_0,x_0).$$
Since $v^+-\varphi $ is upper-semicontintuous and $\overline{B(t_0, x_0, \varepsilon)}-B(t_0, x_0, \varepsilon /2)$ is  compact, this means that there exist a $\delta >0$ such that 
$$\varphi -\delta \geq  v^+\textrm{~on~} \overline{B(t_0, x_0, \varepsilon)}-B(t_0, x_0, \varepsilon /2).$$
Now, if we choose $0<\eta < \delta$ we have that the function 
$$\varphi _{\eta}=\varphi-\eta$$ satisfies the properties
$$-\varphi ^{\eta}_t-\mathcal{L}_t \varphi  ^{\eta} >0\textrm{~on~} \overline{ B(t_0, x_0, \varepsilon)},$$
$$\varphi ^{\eta} >v^+ \textrm{~on~} \overline{B(t_0, x_0, \varepsilon)}-B(t_0, x_0, \varepsilon /2).$$
and 
$$\varphi ^{\eta}(t_0,x_0)=v^+(t_0,x_0)-\eta.$$
Now, we define the new function 
$$v^{\eta}=
\left \{
\begin{array}{l}
 v^+\wedge  \varphi ^{\eta} \textrm{~on~} \overline{ B(t_0, x_0, \varepsilon)},\\
v^+ \textrm{~outside~}\overline{ B(t_0, x_0, \varepsilon)}.
\end{array}
\right.
$$
We clearly have $v^{\eta}$ is upper-semicontinous and $v^{\eta}(t_0,x_0)=\varphi ^{\eta}(t_0,x_0)<v^+(t_0, x_0).$ Also, $v^{\eta}$ satisfies the terminal condition (since $\varepsilon$ can be chosen so that $T>t_0+\varepsilon$ and $v^+$ satisfies the terminal condition).  It only remains to show that 
$v^{\eta}\in \mathcal{U}^+$ to obtain a contradiction. For the analytical Perron  method on viscosity solution, the proof would now be finished, since the viscosity solution  property is \emph{local} and the minimum of two  supersolutions is a supersolution. In our case, the supermartingale property defining $\mathcal{U}^+$ is \emph{global} so we need to localize it using stopping times. Particular care has to be taken since the paths may not be right-continous, so localization in general may fail, as pointed out in Remark \ref{path-regularity}.

Fix $(s,x)$ and 
$\Big((X^{s,x}_t)_{s\leq t\leq T}, (W^{s,x}_t)_{s\leq t\leq T}, \Omega ^{s,x}, \mathcal{F}^{s,x}, \mathbb{P}^{s,x},( \mathcal{F}^{s,x}_t) _{s\leq t\leq T}\Big)\in \mathcal{X}^{s,x}.$
 We need to show that the process $(v^{\eta} (t,X^{s,x}_t))_{s\leq t\leq T}$ is a supermartingale on $(\Omega ^{s,x}, \mathbb{P}^{s,x})$ with respect to the filtration 
$( \mathcal{F}^{s,x}_t) _{s\leq t\leq T}$. We first do the proof under the additional assumption that the process$(v ^+(t,X^{s,x}_t))_{s\leq t\leq T}$ does have RC paths. 

Under this assumption, the process $(v^{\eta} (t,X^{s,x}_t))_{s\leq t\leq T}$  is a supermartingale locally in the region $[s,T]\times \mathbb{R}^d-B(t_0, x_0, \varepsilon/2)$ because it coincides there with the process  $(v^+ (t,X^{s,x}_t))_{s\leq t\leq T}$ which is a RC supermartingale so it can be localized. In addition, in the region $B(t_0, x_0, \varepsilon)$ the process $(v^{\eta} (t,X^{s,x}_t))_{s\leq t\leq T}$ is the minimum between two local supermartingales, therefore a local supermartingale. (It is clear that the process $(\varphi ^{\eta} (t,X^{s,x}_t))_{s\leq t\leq T}$ is a local supermartingale over $B(t_0, x_0, \varepsilon)$ by It\^o's formula.) Since the two regions $[s,T]\times \mathbb{R}^d-B(t_0, x_0, \varepsilon/2)$ and $B(t_0, x_0, \varepsilon)$ actually overlap over an open region, then we can conclude that the process  $(v^{\eta} (t,X^{s,x}_t))_{s\leq t\leq T}$  is indeed a supermartingale.
In order to make this argument, one needs to choose a double sequence of stopping times reminiscent of the optimal strategy in switching control problems. More precisely, the double sequence is chosen as the times  exiting from 
$B(t_0, x_0, \varepsilon)$ and  the  sequel times entering $B(t_0, x_0, \varepsilon/2)$. The choice depends on where the process actually is a time the initial time $s$.

In general, i.e., if the process $(v ^+(t,X^{s,x}_t))_{s\leq t\leq T}$ does not have RC paths, then we can work with its right continuous limit over rational times to reduce it to the case above. More precisely, fix $0\leq s \leq r\leq t\leq T$ and $x\in \mathbb{R}^d$. We want to prove the supermartingale property for the process $(Y_u)_{s\leq u\leq T}:=(v^{\eta} (u,X^{s,x}_u))_{s\leq u\leq T}$ between the times $r$ and $t$, which means we want to show that
\begin{equation}
\label{spermart-r-t}
Y_r\geq \mathbb{E}^{s,x}[Y_t|\mathcal{F}^{s,x}_r].
\end{equation}
First, we make the notation $Z_u:=v ^+(u,X^{s,x}_u)$ for $r\leq u\leq t$ and we stop it at time $t$, i.e.
$Z_u:=v^+ (t,X^{s,x}_t)$ for $t\leq u\leq T$. The process $(Z_u)_{r\leq u\leq T}$ is a supermartingale, but may not be RC, as discussed. We can use Proposition 3.14 page 16  in Karatzas and Shreve \cite{KS88}  to define the RC supermartingale
$$Z^+_u(\omega):=\lim _{q\rightarrow u, q>u,  q\in \mathbb{Q}}Z_q(\omega),\ \ \omega \in \Omega ^*, r\leq u\leq T,$$
and $$Z^+_{\cdot}=0,\ \ \omega \notin \Omega ^*,$$
where $\mathbb{P}^{s,x}[\Omega ^*]=1$.
We would like to emphasize that $Z^+$ is, indeed a RC supermartingale with respect to the original filtration since the filtration is assumed to satisfy the usual conditions. Since the function $v^+$ is USC, and the process is constant after $t$ we can conclude (taking path-wise limits) that 
$$Z_r\geq Z_r^+,\ \ \ Z_t=Z^+_t.$$
We recall that in the \emph{open} region $B(t_0,x_0,\varepsilon)-\overline{B(t_0,x_0,\varepsilon /2)}$
we have $v^+<\varphi -\delta$. Therefore, if we take right limits inside this region, and use the fact that $\varphi$ is continous the we get
 $$Z^+_u<\varphi ^{\eta}(u, X^{s,x}_u),\ \ \textrm{if}\ \ (u,X^{s,x}_u)\in B(t_0,x_0,\varepsilon)-\overline{B(t_0,x_0,\varepsilon /2)}.$$

Now, we can define the process
$$Y^+_u:=
\left \{
\begin{array}{l}
Z^+_u ,\ \ \ \ (u,X^{s,x}_u)\notin \overline{B(t_0,x_0,\varepsilon/2)},\\
Z^+_u \wedge \varphi ^{\eta} (u, x^{s,x}_u),\ \ \ (u,X^{s,x}_u)\in B(t_0,x_0,\varepsilon).
\end{array}\right.$$
We note that we have
$$Y_r\geq Y^+_r,\ \ \ Y_t=Y^+_t.$$
Now, for the process $Y^+$, we can apply the previous argument, since $Z^+$ has RC paths, to conclude it is a supermartingale.
In particular, we have that
$$Y_r\geq Y^+_r\geq \mathbb{E}^{s,x}[Y^+_t|\mathcal{F}^{s,x}_r]= \mathbb{E}^{s,x}[Y_t|\mathcal{F}^{s,x}_r].$$
{\bf Step 2.} \emph{The terminal condition}. Assume that, for some $x_0\in \mathbb{R}^d$ we have 
$v^+(T,x_0)>g(x_0).$
We want to use this information in a similar way to Step 1 to construct a contradiction. Since $g$ is USC on $\mathbb{R}^d$, there exists an $\varepsilon >0$ such that
$$g(x)\leq v^+(T,x_0)-\varepsilon,\ \ \ |x-x_0|\leq \varepsilon.$$
We now use the fact that $v^+$ is USC to conclude it is bounded above on the compact set
$$(\overline{B(T,x_0,\varepsilon)}-B(T,x_0,\varepsilon /2))\cap( [0,T]\times \mathbb{R}^d).$$
This was anyway clear, since actually $v^+$ is globally bounded, but the argument above shows the proof works in even more general cases.  
Now, we choose $\eta>0 $ small enough so that
\begin{equation}
\label{bound}
v^+(T,x_0)+\frac{\varepsilon^2}{4\eta} \geq \varepsilon +\sup_{(t,x)\in (\overline{B(T,x_0,\varepsilon)}-B(T,x_0,\varepsilon /2))\cap( [0,T]\times \mathbb{R}^d)} v^+(t,x).
\end{equation}
We now define, for $k>0$ the following function
$$\varphi^{\eta,\varepsilon, k}(t,x)=v^+(T,x_0)+\frac{|x-x_0|^2}{\eta}+ k(T-t). $$
For $k$ large enough we have that 
$$-\varphi^{\varepsilon, \eta,k}_t-\mathcal{L}_t \varphi ^{\varepsilon, \eta, k}>0 ,\ \ \ \textrm{on~} \overline{B(T,x_0,\varepsilon)}.$$
In addition, using \eqref{bound} we have that
$$\varphi^{\varepsilon, \eta,k}\geq \varepsilon +v^+,\textrm{~on~} (\overline{B(T,x_0,\varepsilon)}-B(T,x_0,\varepsilon /2))\cap( [0,T]\times \mathbb{R}^d).$$
Also, $\varphi^{\varepsilon, \eta,k}(T,x)\geq v^+(T,x_0)\geq g(x)+\varepsilon$ for $|x-x_0|\leq \varepsilon$. Now, we can choose $\delta <\varepsilon$ and define as in the proof of Step 1
$$v^{\varepsilon, \eta,k, \delta }=
\left \{
\begin{array}{l}
 v^+\wedge \Big ( \varphi ^{\varepsilon, \eta,k}-\delta  \Big)\textrm{~on~} \overline{ B(T, x_0, \varepsilon)},\\
v^+ \textrm{~outside~}\overline{ B(T, x_0, \varepsilon)}.
\end{array}
\right.
$$
We can now prove, using the same switching principle and RC modification argument as in Step 1 that $v^{\varepsilon, \eta,k, \delta }\in \mathcal{U}^+,$ but $v^{\varepsilon, \eta,k, \delta }(T,x_0)=v^+(T,x_0)-\delta$, leading to a contradiction.
\end{proof}

\subsection{Verification by comparison}
\begin{Definition}\label{comparison} We say that the viscosity comparison principle holds for the equation \eqref{pde} with respect to time horizon $T$ and the final condition $g$, or that  condition $CP(T,g)$ is satisfied if,
whenever we have  a bounded, upper-continuous (USC) subsolution $u$ of  
\begin{equation}\label{pde3}
\left \{
\begin{array}{l}
-u_t-L_tu\leq  0,\\
u(T,x)\leq g(x),
\end{array}\right.\end{equation}
and a bounded lower semicontinous super-solution $v$ of 
 \begin{equation}\label{pde4}
\left \{
\begin{array}{l}
-u_t-L_tu\geq 0,\\
u(T,x)\geq g(x).
\end{array}\right.\end{equation}
 then 
$u\leq v.$
\end{Definition}
Next theorem is an easy consequence of our main result, Theorem \ref{Perron}. However, it   amounts to a verification result for non-smooth viscosity  solutions of \eqref{pde}, so we consider it to be the other main result of the present note.
\begin{Theorem}\label{theorem}
Let $g$ be bounded and continous. Assume also that the comparison principle  $CP(T,g)$ is satisfied. Then  there exists a unique bounded and continuous viscosity solution $v$ to \eqref{pde} which equals both the lower and the upper  pay-offs (value functions), which means
$$v_*=v=v^*.$$
In addition,
 for each $(s,x)\in[0,T]\times \mathbb{R}^d$, and each weak solution 
  $$\Big((X^{s,x}_t)_{s\leq t\leq T}, (W^{s,x}_t)_{s\leq t\leq T}, \Omega ^{s,x}, \mathcal{F}^{s,x}, \mathbb{P}^{s,x},( \mathcal{F}^{s,x}_t) _{s\leq t\leq T}\Big)\in \mathcal{X}^{s,x},$$ 
the process
$(v(t,X^{s,x}))_{s\leq t\leq T}$ is a martingale  on $(\Omega ^{s,x}, \mathbb{P}^{s,x})$ with respect to the filtration 
$( \mathcal{F}^{s,x}_t) _{s\leq t\leq T}$.
\end{Theorem}
\begin{proof}\mbox{} The proof is immediate in light of Definition \ref{comparison}, Lemma \ref{easy-Perron} and Theorem \ref{Perron}.
\end{proof}
 \begin{Remark}\label{Markov} The martingale property for $(v(t,X^{s,x}))_{s\leq t\leq T}$ is proved without using the Markov property of the weak solution (which is  not even  assumed). 
  However, if $CP(T,g)$ is satisfied for any $T$ and any bounded  (test function) $g$, then we obtain that,
  for each $(s,x)$ and each $T$ the law of $X^{s,x}_T$ is uniquely determined.
  Following Theorem 6.2.3 in \cite{MR2190038} or Proposition 4.27 page 326 in \cite{KS88},  
  uniqueness of marginals implies   uniqueness in law of the weak solution. Now, uniqueness in law for any $(s,x)$ does imply the  Markov  property for the weak solution of the SDE \eqref{sde} (the Markov property holds with respect to the natural raw filtration though). We refer the reader to Theorem 6.2.2 in  \cite{MR2190038} or Theorem 4.20 page 322 in \cite{KS88} for the last mentioned result.
 \end{Remark}
 \begin{Remark}\label{selection} The whole paper can be rewritten by selecting, for each $(s,x)$, only one weak solution $X^{s,x}$  instead of using  the 
  set of weak solutions $\mathcal{X}^{s,x}$. 
  Such a selection uses the axiom of choice and does not need to be a Markov selection.
  Once the selection $X^{s,x}$ is chosen,  the sets of stochastic super and sub-solutions in Definitions \ref{def:supersolutions} and \ref{def:subsolutions} have to be re-defined accordingly, and there is only one pay-off (value function) $v$  replacing $v_*$ and $v^*$.
 \end{Remark}
 \section{Conclusions} We designed a stochastic counterpart to Perron's method which produces two viscosity solutions of the Feynman-Kac equation. The two solutions squeeze in between the expected pay-off, and this comparison is a trivial consequence of the probabilistic definition. If, in addition, a viscosity comparison result holds, then we do have a unique viscosity solution, which is  a martingale along the solutions of the stochastic differential equation and is equal to the expected pay-off.  In this case, we therefore have a full verification result without smoothness of the viscosity solution.
 
 While the Perron method we describe here is reminiscent of the characterization of the value function in optimal stopping problems as the least excessive function, we would like to point out that here, unlike in optimal stopping, we avoid proving that the value function is ``excessive". This is actually the point of verification by comparison, to avoid working with the value function. 
 
 One could try to prove directly, avoiding viscosity altogether, that $v^+$ and  $v^-$ along solutions of the SDE are  martingales, i.e.  they are stochastic solutions in the sense of Stroock and Varadhan \cite{Stroock-Varadhan}. However, this 
 is actually not  possible: $v^+$ and $v^-$ are stochastic solutions \emph{only} when they coincide, since the stochastic solutions are unique by definition. We can additionally justify  that, in general, 
  $v^+$ and $v^-$ are viscosity solutions but \emph{may not} be  stochastic solutions  by observing 
  that  the stochastic solution property is much stronger then viscosity. Fortunately, in a large number of situations, viscosity property is still strong enough to prove uniqueness. In this case, viscosity and stochastic solution property are equivalent.
 \section{Appendix: Countable selection to achieve the  inf/sup of a class of semi-continous functions}\label{appendix}
 The main purpose of the Appendix is to prove a countable selection argument needed in the proof of Lemma \ref{lem:verification}.  Let $(M,d)$ be a metric space and consider a class $ \mathcal{G}$  of functions $f:M\rightarrow \overline{\mathbb{R}}$. 
 The first result is 
 \begin{Lemma}\label{epigraph}Let $g:M\rightarrow \overline{\mathbb{R}}.$ Then, the following conditions are equivalent
 \begin{enumerate}
 \item $g(x)=\inf _{f\in \mathcal{G}}f(x)$, for each $x\in M$,
 \item $\{x\in M |g(x)<q\}=\cup _{f\in \mathcal{F}}\{ x\in M |f(x)<q\}$, for each $q\in \mathbb{Q}$.
 \end{enumerate}
 \end{Lemma}
 \begin{Proposition}\label{countable} Assume that $(M,d)$ is a separable metric space (or, less, a topological space with a countable base). Assume also that each function in the class $\mathcal{G}$ is upper-semicontinous (USC). Then, there exists a countable subclass of functions $\mathcal{H}\subset \mathcal{G}$ such that
 $$f_*(x):=\inf _{f\in \mathcal{G}}f(x)=\inf _{f\in \mathcal{H}}f(x) ,\ \ 
 \textrm{for~each~} x\in M.$$ \end{Proposition}
\begin{proof}\mbox{} Fix a $q\in \mathbb{Q}.$ According to Lemma \ref{epigraph}, the \emph{open} set
$\{x\in M|f_*(x)<q\}$ admits an \emph{open} cover as  
$$\{x\in M|f_*(x)<q\}=\cup _{f\in \mathcal{G}}\{ x\in M |f(x)<q\}.$$
Since the space $(M,d)$ is separable, so it admits a countable basis, one can select a \emph{countable} open sub-cover. 
More precisely, there exists a \emph{countable} $\mathcal{G}_q\subset \mathcal{G}$ such that
$$\{x\in M |f_*(x)<q\}=\cup _{f\in \mathcal{G}_q}\{ x\in M |f(x)<q\}.$$
Now, we define the countable class
$$\mathcal{H}:=\cup _{q\in \mathbb{Q}}\mathcal{G}_q.$$
We have, for each $q\in \mathbb{Q}$ that
\[
\begin{split}
 \{x\in M |f_*(x)<q\}&=\cup _{f\in \mathcal{G}_q}\{ x\in M |f(x)<q\}\subset 
\cup _{f\in \mathcal{H}}\{ x\in M |f(x)<q\}
 \\ 
&\subset  \{x\in M |f_*(x)<q\}.
\end{split}
\]
According to Lemma \ref{epigraph} we then have that 
$$f_*(x)=\inf _{f\in \mathcal{H}}f(x),\ \ \ \textrm{for each}\ x\in M.$$
\end{proof}
\begin{Remark}\label{axiomofchoice} We would like to point out that in the argument of selecting a countable open sub-cover the axiom of choice is used.
\end{Remark}
\bibliographystyle{amsplain}

\bibliography{references} 

\end{document}